\documentclass[12pt]{amsart}
\usepackage{amscd}
\usepackage{amssymb}
\usepackage{a4wide}
\usepackage{amstext}
\usepackage{amsthm}
\usepackage{mathrsfs}
\usepackage[final]{hyperref}
\hypersetup{unicode= false, colorlinks=true, linkcolor=blue,
anchorcolor=blue, citecolor=green, filecolor=red, menucolor=blue,
pagecolor=blue, urlcolor=blue} 
\numberwithin{equation}{section}

 \DeclareMathOperator{\dist}{dist}

 \DeclareMathOperator{\sgn}{sgn}
\DeclareMathOperator{\supp}{supp} 

\DeclareMathOperator{\spec}{spec} 
\DeclareMathOperator{\osc}{osc} 
\DeclareMathOperator{\Ker}{Ker} 
\DeclareMathOperator{\Lip}{Lip} 

\renewcommand{\phi}{\varphi}

\newcommand{\pw}{\mathcal{P}W_\pi}

\newcommand{\he}{\mathcal{H}(E)}

\newtheorem{Thm}{Theorem}[section]
\newtheorem{theorem}[Thm]{Theorem}
\newtheorem{lemma}[Thm]{Lemma}
\newtheorem{proposition}[Thm]{Proposition}
\newtheorem{corollary}[Thm]{Corollary}
\newtheorem{remark}[Thm]{Remark}
\newtheorem{definition}[Thm]{Definition}

\begin{document}
\sloppy
\title[]{The Beurling--Malliavin Multiplier Theorem \\ and its analogs for the de Branges spaces}
\author{Yurii Belov, Victor Havin}
\address{Yurii Belov,
\newline 
Chebyshev Laboratory,
St.Petersburg State University,
St.Petersburg, Russia,
\newline {\tt j\_b\_juri\_belov@mail.ru}
\newline\newline \phantom{x}\,\, 
Victor Havin, 
\newline
Department of Mathematics and Mechanics,
St.Petersburg State University,\hfill\hfill\\
St.Petersburg, Russia,
\newline {\tt havin@VH1621.spb.edu}
}
\thanks{The first author was supported by the Chebyshev Laboratory 
(St. Petersburg State University) under RF Government grant 11.G34.31.0026,
by JSC "Gazprom Neft", and by RFBR grant 12-01-31492. The second author was supported by St. Petersburg State University Action Item 2: NIR "Function theory, operators theory and its applications" 6.38.78.2011.}

\begin{abstract}
Let $\omega$ be a non-negative function on $\mathbb{R}$. 
We are looking for a non-zero $f$ from a given space of entire functions $X$ satisfying 
$$(a) \quad|f|\leq \omega\text{\quad or\quad(b)}\quad |f|\asymp\omega.$$
The classical Beurling--Malliavin Multiplier Theorem corresponds to $(a)$ and the classical
Paley--Wiener space as $X$. We survey recent results for the case 
when $X$ is a de Branges space $\he$. Numerous answers mainly depend on the behaviour of the 
phase function of the generating function $E$.
\end{abstract}

\maketitle

This survey article consists of two parts. The first (Section \ref{s1}) is devoted to the Beurling--Malliavin Multiplier Theorem (the BM-theorem):
\begin{theorem} If $\int_\mathbb{R}\frac{\Omega(x)}{1+x^2}dx<\infty$ where $\Omega$ is a non-negative Lipschitz function on $\mathbb{R}$, then for any $\sigma>0$ there exists a non-zero function $f\in L^2(\mathbb{R})$ such that its Fourier transform vanishes on $\mathbb{R}\setminus[-\sigma,\sigma]$ and $|f|\leq e^{-\Omega}$.
\label{mainth}
\end{theorem}

The term "multiplier" is explained in Subsection \ref{42}.

This deep and difficult result has important connections with problems of harmonic and complex analysis. Published in 1962 (see \cite{BM}) it remains topical even nowadays. The second part (Sections \ref{s2} and \ref{s3}) of this article describes recent analogs of the BM-theorem related to the de Branges spaces of entire functions. 

\textbf{Acknowledgement.} We are grateful to A. Borichev for the permission to expose his construction illustrating the sharpness of the BM-theorem (see Section \ref{constr} below).

\section{On the Beurling--Malliavin Multiplier Theorem\label{s1}}

\subsection{Bounded and semibounded spectra.\label{hardy}}  For a Lebesgue measureable function $f: \mathbb{R}\rightarrow\mathbb{C}$ we denote by $\supp f$ its (closed) support, i.e. $\supp f:= \mathbb{R}\setminus O_f$, where $O_f$ is the union of all open $O$'s such that $f=0$ a.e. on $O$. For $f\in L^2(\mathbb{R})(=L^2)$ we put $\spec f:=\supp \hat{f}$, $\hat{f}$ being the Fourier transform of $f$, $\hat{f}(x)=\int_{\mathbb{R}}f(t)e^{-itx}dx$, $x\in\mathbb{R}$, defined as in the Plancherel theorem. The set $\spec f$ is called {\it the spectrum of } $f$.

A subset $E$ of a ray $(-\infty, a]$ or $[a,+\infty)$, $a\in\mathbb{R}$, is called {\it semibounded}. Put
$$H^2(\mathbb{R})(=H^2):=\{f\in L^2: \spec f\subset[0,+\infty\}.$$
Recall that $H^2(\mathbb{R})$ is a close relative of the Hardy class $H^2(\mathbb{C_{+}})$ of functions $F$ analytic in the upper half-plane $\mathbb{C}_+$ and such that $\sup_{y>0}\int_{\mathbb{R}}|F(x+iy)|^2dx<\infty$. Namely, $H^2(\mathbb{R})$ is the set of all boundary traces of functions $F\in H^2(\mathbb{C}_+)$:
$$f\in H^2(\mathbb{R})\Leftrightarrow f(x)=\lim_{y\rightarrow0+}F(x+iy) \text{ a.e. for  an } F\in H^2(\mathbb{C}_+),$$
and
$$\lim_{y\rightarrow 0+}\int_{\mathbb{R}}|f(x)-F(x+iy)|^2dx=0.$$

$L^2$-functions with bounded spectra also admit a complete description by means of analytic functions. To see this we need the Paley--Wiener class $\mathcal{PW}_\sigma$, $\sigma>0$, of all entire functions $F$ such that 
$$\text{ (a) } F\bigr{|}_{\mathbb{R}}\in L^2\quad\text{ and }\quad\text{(b) } |F(z)|\leq C_Fe^{\sigma|z|},\quad z\in\mathbb{C}.$$
We get an equivalent definition replacing $\sigma|z|$ in (b) by $\sigma|\Im z|$ (see \cite[p. 175]{HJ}).

Now, {\it the following assertions are equivalent for an $L^2$-function $f$ and $\sigma>0$:} 
\begin{enumerate}
\begin{item}
$\spec f\subset[-\sigma, \sigma]$;
\end{item}
\begin{item}
$f$ {\it coincides a.e. on  $\mathbb{R}$ with an $F\in\mathcal{PW}_{\sigma}$.}
\end{item}
\end{enumerate}
This is the famous Paley--Wiener theorem (\cite[p. 174]{HJ}).
\subsection{BM-majorants and the logarithmic integral.} Let $\omega$ be a bounded non-negative function on $\mathbb{R}$. We call it {\it a Beurling--Malliavin majorant} (= $BM$-majorant) and write $\omega\in BM$ if for any $\sigma>0$ there exists a non-zero $f\in L^2$ such that
\begin{equation}
\text{(a) } |f|\leq\omega,\quad \text{(b) }\spec f\subset[-\sigma, \sigma].
\end{equation}
The properties (a) and (b) mean that $f=F\bigr{|}_{\mathbb{R}}$ a.e. for an $F\in\mathcal{PW}_{\sigma}$.
\subsubsection{} To explain the origin, the meaning and the interest of the class $BM$ we need so-called logarithmic integrals. For a Lebesgue measurable function $f:\mathbb{R}\rightarrow\mathbb{C}$ put 
\begin{equation}
\mathcal{L}(f):=\int_{\mathbb{R}}\frac{\log|f(x)|}{1+x^2}dx.
\end{equation}
We call $\mathcal{L}(f)$ {\it the logarithmic integral of } $f$. It makes sense for any $f\in L^2$. In this case $\mathcal{L}(f)<+\infty$, but the equality $\mathcal{L}(f)=-\infty$ is not excluded. It expresses sort of smallness of $f$ and, in particular, may be caused by vanishing of $f$ on a set of positive length or by a fast decay of $|f(x)|$ as $x$ tends to a finite or infinite limit.
\subsubsection{\label{122}}
The following fact is crucial for our theme: {\it for an $L^2$-function $f$ with a semibounded spectrum}
\begin{equation}
\mathcal{L}(f)=-\infty\Rightarrow f = 0 \text{ a.e. }
\label{log}
\end{equation}
(see, e.g., \cite[Part Two, Ch.2]{HJ}) This result is one of innumerable manifestations of the Uncertainty Principle (UP for short) 
forbidding a simultaneous and excessive smallness of a non-zero $f$ and $\hat{f}$ (see, \cite{Koo1, Koo2, Koo3, H, HJ, HMN}, the literature on the UP is very numerous). The smallness of $f$ and $\hat{f}$ in \eqref{log} is expressed by the equalities $\mathcal{L}(f)=-\infty$ and $\hat{f}\bigr{|}_I\equiv0$ where $I$ is a ray.

The implication \eqref{log} is sharp. Moduli of $H^2$-functions admit a complete and simple description: {\it a non-negative non-zero $L^2$-function $\phi$ is the modulus of an $f\in H^2$ if and only if $\mathcal{L}(\phi)>-\infty$; $f$ can be defined by the formula
\begin{equation}
f(x)=\lim_{y\rightarrow0}\exp \biggl{[}\frac{1}{\pi i}\int_{\mathbb{R}}\frac{\log\phi(t)}{t-(x+iy)}\cdot\frac{1+t(x+iy)}{1+t^2}dt\biggr{]}:=O_\phi(x)
\label{out}
\end{equation}
for almost all $x\in\mathbb{R}$} (see, e.g., \cite[sect. 3.6.5]{HMN}). Function $f$ defined by \eqref{out} is called {\it the outer function 
corresponding to } $\phi$. Moreover any non-zero function $f\in H^2$ admits  a representation of the form $O_{|f|}I$, $I$ is an {\it inner} function in $\mathbb{C}_+$ (a bounded analytic function in $\mathbb{C}_+$ with the unimodular trace a.e. on $\mathbb{R}$). This representation is called {\it inner-outer factorization } of $f$.

\subsubsection{} The UP suggests the following question: how small a non-zero $L^2$-function with a {\it bounded} (not just {\it semi}bounded) spectrum can be? The $L^2$ functions with bounded spectra are much "more analytic" than $H^2$-functions, i.e. the boundary traces of $H^2(\mathbb{C}_+)$-functions. This fact complicates the quest of an appropriate form of the UP. The definition of a $BM$-majorant is dictated by this problem. Clearly, the convergence of the integral $\mathcal{L}(\omega)$ is {\it necessary} for a majorant to be in $BM$. But (unlike the case of {\it semi}bounded spectra) it is not sufficient anymore.

There exist non-negative bounded and continuous $\omega$'s with $\mathcal{L}(\omega)>-\infty$, but not in $BM$. To see this consider a bounded interval $I$ with length $|I|$ and center $c(I)$ and put $\phi_I(t):=\frac{2|t-c(I)|}{|I|}$, $t\in I$. For a sequence $\{I_n\}^\infty_{n=1}$ of bounded and pairwise disjoint intervals with $c(I_n)\rightarrow\infty$ as $n\rightarrow\infty$ put
$$\omega(t)=\phi_{I_n}(t)\text{ for } t\in I_n, n=1,2,...,\quad \omega(t)=1 \text{ elsewhere.}$$
Suppose $c(I_n)= \sqrt{n}$. Then $\omega\notin BM$. Indeed, consider a non-zero $F\in\mathcal{PW}_\sigma$ for a $\sigma>0$. Then the number $n(r)$ of zeros of $F$ in the big disc $\{|z|<r\}$ is $O(r)$ (by the Poisson-Jensen inequality), so that the estimate
 $|F|\leq\omega$ on $\mathbb{R}$ is impossible (since $\omega(\sqrt{n})\equiv0$). But at the same time $\mathcal{L}(\omega)=\sum_{n=1}^\infty\int_{I_n}\frac{\log\phi_{I}(t)}{1+t^2}dt>-\infty$ if, say, $\sum_{n=1}^\infty|I_n|<+\infty$.

This argument can be changed slightly to provide a {\it strictly} positive continuous $\omega$ with $\mathcal{L}(\omega)>-\infty$, but not in $BM$ (see \cite{HJ,HMN}).

Note that the walls of the pits on the graph of $\omega$ (i.e. graphs of $\phi_{I_n}$) in the above construction are bound to get more and more steep as $n$ grows. As we will see below, a majorant $\Omega:=-\log\omega$ with $\mathcal{L}(\omega)>-\infty$ and not in $BM$ cannot be Lipschitz. On the other hand its slope $|\Omega'|$ may grow arbitrarily slow as is shown in the next subsection.

We conclude with the following remark: if a majorant $\omega$ does not oscillate, then the convergence of $\mathcal{L}(\omega)$ is not only necessary, but also {\it sufficient} for $\omega$ to {\it be in $BM$}. To be more precise: {\it if a positive $\omega$ is monotone on $(-\infty,0]$ and $[0,+\infty)$, then $\mathcal{L}(\omega)>-\infty\Rightarrow\omega\in BM$.}

This theorem has several proofs (see, e.g. \cite{Koo1,HJ,H,HMN}) and has applications to weighted polynomial approximation and quasianalyticity.

\subsection{More on the oscillations of $BM$-majorants. Borichev's construction.\label{constr}}

In this subsection we expose another approach to majorants with a finite logarithmic integral, but not in $BM$ \cite{Bor}. The result of Subsection \ref{mainBor} shows (in particular) that given an increasing and unbounded $H:\mathbb{R}\rightarrow(0,+\infty)$ there exists an $\Omega\in C^1(\mathbb{R})$ such that $\Omega>0$, $\mathcal{L}(e^{-\Omega})>-\infty$, $|\Omega'|\leq H$, $e^{-\Omega}\notin BM$. This is impossible if $H$ is bounded (by the BM multiplier theorem). Subsections \ref{Bor1}-\ref{Bor2} are preparatory.

\subsubsection{\label{Bor1}} Suppose $\psi:\mathbb{R}\rightarrow\mathbb{C}$ is a Lebesgue measurable function and $\int_{\mathbb{R}}\frac{|\psi(x)|}{1+x^2}dx<+\infty$. We denote by $v_{\psi}$ its harmonic extension to the upper half-plane $\mathbb{C}_+$, i.e. 
$$v_{\psi}(z):=\frac{1}{\pi}\int_{\mathbb{R}}\frac{\Im z}{|t-z|^2}\psi(t)dt,\quad \Im z>0.$$
For a compact interval $I\subset\mathbb{R}$ and $x\in\mathbb{R}$ put $T_I(x):=\dist(x,\mathbb{R}\setminus I)$, a "solitary tooth" of height $|I|\slash2$ based on $I$; $|I|$ stands for the length of $I$. We put $v_I:=v_{T_I}$, $v:=v_{[-1,1]}$; $v_I$ is continuous in $\mathbb{C}_+\cup\mathbb{R}$ and strictly positive in $\mathbb{C}_+$. Clearly $v_I(z)=\frac{1}{2}v\bigl{(}\frac{2(z-c(I))}{|I|}\bigr{)}$, $\Im z \geq0$ (the left and right sides are the Poisson integrals and coincide on $\mathbb{R}$, $c(I)$ is the center of $I$). For a positive $\sigma$ we denote by $\mathcal{E}_{\sigma,1}$ the set of all entire functions $f$ such that 
$$|f(z)|\leq e^{\sigma |z|} \text{ for any } z\in\mathbb{C},\quad |f|\leq 1 \text{ on }\mathbb{R}.$$
Thus $\mathcal{E}_{\sigma,1}$ is invariant under real shifts $z\mapsto z+x$, $x\in\mathbb{R}$ of the argument $z$.
\subsubsection{\label{Bor2}}
The smallness of a function $f\in\mathcal{E}_{\sigma,1}$ is contagious: if $|f|$ is small on an interval it is also small on a much larger concentric interval. This is shown by the next lemma.

\begin{lemma}
For any $\sigma>0$ there exist a (small) $\alpha(\sigma)\in(0,1\slash2)$ and a (big) $h(\sigma)>2$ such that for any $h\geq h(\sigma)$, any $f\in\mathcal{E}_{\sigma,1}$ and any compact interval $I\subset\mathbb{R}$
$$|f|\leq e^{-hT_I}\text{ on }\mathbb{R} \Rightarrow |f|\leq e^{-Ch|I|} \text{ on } \tilde{I}=I_{\alpha(\sigma),h},$$
where $C>0$ is an absolute constant and $\tilde{I}$ is the interval centered at $c(I)$ with $|\tilde{I}|=\sqrt{h^{2\alpha(\sigma)}-1}|I|$.

Note that 
$$\frac{1}{2}h^{\alpha(\sigma)}|I|\leq|\tilde{I}|\leq h^{\alpha(\sigma)}|I|$$
if $h\geq h(\sigma)$ (for big values of $h(\sigma)$).
\label{mainlemma}
\end{lemma}
\begin{proof}
Suppose $f\in\mathcal{E}_{\sigma,1}$, $f\neq0$ and $|f|\leq e^{-hT_I}$ on $I$, $h>1$. Then $v_{\log|f|}$ makes sense (see \cite[p.306]{HJ}). Put $f_{\sigma}(z):=f(z)e^{i\sigma z}$, so that $f_\sigma$ is bounded in $\mathbb{C}_+$ whence
\begin{equation}
\log|f_{\sigma}(z)|\leq v_{\log|f|}(z)\leq -hv_{I}(z),\quad z\in\mathbb{C}_+
\label{logeq}
\end{equation}
We may assume $c(I)=0$ whence $v_I(z)=\frac{|I|}{2}v\bigl{(}\frac{2|z|}{|I|}\bigr{)}$, $z\in\mathbb{C}_+$. Consider three closed concentric disks $D_j, j=1,2,3$ centered at $i\frac{|I|}{2}$ and of radii $R_1=\frac{|I|}{2}, R_2=h^\alpha R_1, R_3=hR_1$
where $\alpha\in(0,1)$ depends on $\sigma$ and will be chosen later. For $z\in D_1$ the point $\frac{2z}{|I|}$ is in the closed disc $d_1$ of radius one centered at $i$. Hence, $v_I(z)\geq c|I|$, $c:=\min_{d_1}v>0$ (note that $v$ is strictly positive on $d_1$). Thus by \eqref{logeq}
$$|f(z)|\leq|e^{-i\sigma z}|e^{-ch|I|}\leq e^{(-ch+\sigma)|I|}\leq e^{-ch|I|\slash2},\quad z\in D_1$$
provided $h\geq \frac{2\sigma}{c}$. Now, $|f(z)|\leq e^{\sigma|\Im z|}\leq e^{\sigma h|I|}$ for $z\in D_3$, and, by the Hadamard three circles theorem
\begin{equation}
\max_{D_2}|f|\leq e^{[-(1-\alpha)C'+\alpha\sigma]h|I|},\quad C'=\frac{C}{2}, h\geq\frac{\sigma}{C'}
\label{D2eq}
\end{equation}
Put $\alpha(\sigma):=\frac{C'}{2(C'+\sigma)}$. Then \eqref{D2eq} becomes 
$$\max_{D_2}|f|\leq e^{-Ch|I|},\quad C:=\frac{C'}{2}$$
whereas $D_2\supset I_{\alpha(\sigma), h}=:\tilde{I}$(= the chord of $D_2$ lying in $\mathbb{R}$).
\end{proof}

\subsubsection{\label{mainBor}}
\begin{theorem}
Let $H$ be a positive function on $\mathbb{R}$ increasing and unbounded on $[0,+\infty)$. Then there exists a non-negative $\Omega\in C^1(\mathbb{R})$ such that
\begin{enumerate}
\begin{item}
$|\Omega'|\leq H$,
\end{item}
\begin{item}
$\int_{\mathbb{R}}\frac{\Omega(t)}{1+t^2}dt<+\infty,$
\end{item}
\begin{item}
$f\in L^2$, $\spec f$ is bounded, $|f|\leq e^{-\Omega} \Rightarrow f= 0 \text{ a.e. }$
\end{item}
\end{enumerate}
\label{borth}
\end{theorem}

Clearly, $e^{-\Omega}\notin BM$.
\begin{proof}
We prove a slightly weaker assertion providing a piecewise linear continuous $\Omega$ enjoying $(ii)$ and $(iii)$ with $(i)$ fulfilled outside a sparse discrete set, so that regularization is needed to get $(i)$ everywhere. The graph of our $\Omega$ will be a saw with very high and rare sawteeth:
$$\Omega:=\sum_{k=1}^\infty h_kT_k,\quad T_k:=T_{I_k},\quad h_k:=H(x_k),$$
$$I_k:=[x_k,x_k+|I_k|],\quad k=1,2,...,$$
$I_k$ being pairwise disjoint, $k=1,2,...$.  We choose $x_k$ to get 
$$2x_k<x_{k+1},\qquad h_k>k,\quad k=1,2,...$$
The lengths $|I_k|$ are defined by the equalities $(|I_k|^2h_k)x^{-2}_k=\frac{1}{(k+3)\log^2(k+3)}$, so that 
$s:= \sum_{k=1}^{\infty}\frac{|I_k|^2}{x^2_k}h_k<1$ whence
$$\max I_k=x_k+|I_k|<x_k(1+h^{-1\slash2}_k)<2x_k<x_{k+1}=\min I_{k+1}$$
and 
$$\int_{\mathbb{R}}\frac{\Omega(x)}{1+x^2}dx=\sum_{k=1}^\infty\int_{I_k}\frac{h_kT_k(x)}{1+x^2}dx\leq\sum_{k=1}^\infty h_k\frac{|I_k|^2}{x^2_k}<+\infty,$$
and $(ii)$ follows; $(i)$ is obvious on $(x_k, x_k+|I_k|\slash2)\cup(x_k+|I_k|\slash2, x_k+|I_k|)$  by the monotonicity of $H$ on $[0,+\infty)$, and thus it is true everywhere on $\mathbb{R}$ (except for the ends and centers of $I_k$'s). Turn to $(iii)$ and fix a $\sigma>0$ and $f\in\mathcal{E}_{\sigma,1}$ such that $|f|\leq e^{-\Omega}$. The intervals $\tilde{I_k}=\tilde{I_k}(\sigma)$ (see Lemma \ref{mainlemma}) do not overlap for $k\geq k(\sigma)$. Indeed, $s<1$, and therefore 
$$\max \tilde{I_k}(\sigma)\leq x_k+|\tilde{I_k}|\leq x_k+h^{\alpha(\sigma)}_k|I_k|\leq x_k(1+h^{\alpha(\sigma)-1\slash2}_k),$$
whereas 
$$\min \tilde{I}_{k+1}(\sigma)\geq x_{k+1}-|\tilde{I}_{k+1}(\sigma)|\geq x_{k+1}(1-h^{\alpha(\sigma-1\slash2)}_{k+1})>2x_k(1-h^{\alpha(\sigma)-1\slash2}_{k+1})$$
(recall $\alpha(\sigma)<1\slash2$). Thus $\max\tilde{I_k}(\sigma)<\min \tilde{I}_{k+1}(\sigma)$ for $k\geq k(\sigma)$. Hence
by Lemma \ref{mainlemma}
$$\int_1^\infty\frac{\log|f(x)|}{x^2}dx\leq \sum_{k\geq k(\sigma)}\int_{\tilde{I_k}(\sigma)}\frac{\log|f(x)|}{x^2}dx$$
$$\leq-C\sum_{k\geq k(\sigma)}\frac{|I_k|h_k}{x^2_k}h^{\alpha(\sigma)}_k|I_k|=
-C\sum_{k\geq k(\sigma)}\frac{k^{\alpha(\sigma)}}{(k+3)\log^2(k+3)}=-\infty,$$
and so $f\equiv 0$.
\end{proof}

\begin{remark}
The construction of $\Omega$ can be generalized as follows: for a positive sequence $\{t_k\}_{k=1}^\infty$ with $\sum_{k=1}^\infty t_k<1$ find an increasing sequence $\{b_k\}_{k=1}^\infty$ such that $b_1>1$, $\lim_{k\rightarrow\infty}b_kt^N_k=+\infty$ for any $N=1,2,...$; then choose $x_k$ so that $h_k:=H(x_k)>b_k$, $x_{k+1}>2x_k$, $k=1,2,...$, and define $|I_k|$ by $(|I_k|^2\slash x^2_k)h_k=t_k$.
\end{remark}

\subsection{Some reformulations of the BM-Theorem.} In this subsection $\omega$ denotes a function continuous on $\mathbb{R}$ and such that 
\begin{equation}
0<\omega\leq1,\quad \mathcal{L}(\omega)>-\infty.
\label{omega}
\end{equation}
We put $\Omega:=\log\frac{1}{\omega}$ and denote by $\osc_I\Omega$ the oscillation of $\Omega$ on the interval $I$: $$\osc_I(\Omega):=\sup\{\Omega(x)-\Omega(y): x,y\in I\}.$$

\subsubsection{\label{41}}
The following statement is equivalent to the BM-theorem.
\begin{theorem}
Suppose $\omega$ satisfies \eqref{omega} and the oscillations of $\Omega$ on intervals of length one are uniformly bounded, i.e. $C(\omega):=\sup_{|I|\leq1}\osc_{I}\Omega<+\infty$. Then $\omega\in BM$.
\end{theorem}
\begin{corollary} If $\omega$ satisfies \eqref{omega} and $\Omega$ is uniformly continuous, then $\omega\in BM$.
\end{corollary}

The deduction of these facts from the BM-theorem is quite simple. Put $\Omega_1(x):=\int_{x-1\slash2}^{x+1\slash2}\Omega(t)dt$, $x\in\mathbb{R}$. Then $\Omega'_1$ and $|\Omega-\Omega_1|$ do not exceed $C(\omega)$ whence $\Omega_1\in \Lip_1(\mathbb{R})$, $\omega_1\slash\omega$ is separated from zero and infinity, so that $\mathcal{L}(\omega_1)>-\infty$, $\omega_1\in BM$ by the BM-theorem, and $\omega\in BM$.

\subsubsection{\label{42}} Here we explain the presence of the term "multiplier" accompanying the BM-theorem.

We say that an entire function $F$ belongs to the Cartwright class and write $F\in Cart$ if 
\begin{enumerate}
\begin{item}
$\int_\mathbb{R}\frac{\log_+|F(x)|}{1+x^2}dx<+\infty,$
\end{item}
\begin{item}
$|F(z)|=O(e^{\sigma|z|}), |z|\rightarrow+\infty$ for a $\sigma>0$.
\end{item}
\end{enumerate}
This class turns out to be useful in Complex and Harmonic analysis (see \cite{Koo1, L, HJ, Lev}). Note that $\mathcal{PW}_\sigma\cup\mathcal{E}_{\sigma,1}\subset Cart$ and $\mathcal{L}(F)$ is finite for any non-zero $F\in Cart$. According to a Krein theorem (\cite[p.192]{HJ}) $Cart$ coincides with the class of all entire functions whose restrictions to the upper and lower half-planes are quotients of functions analytic and bounded in the respective half-plane. The following theorem of Beurling and Malliavin is "parallel" to the Krein theorem.
\begin{theorem}
The Cartwright class coincides with the class of quotients $A\slash B$ where $A,B\in\mathcal{E}_{\sigma,1}$ for a $\sigma>0$. Moreover, for any $\varepsilon>0$ and any $F\in Cart$ there is a $\phi \in\mathcal{E}_{\varepsilon,1}$, $\phi\neq0$ ("a multiplier") such that $\phi F$ is bounded on $\mathbb{R}$.
\label{mult}
\end{theorem}
This result is {\it equivalent } to Theorem \ref{mainth}. This was proved by Koosis (\cite{Koo2}). To sketch his argument let us say that, given classes $M$ and $N$ of functions defined and positive on $\mathbb{R}$, $M$ {\it minorizes} $N$ if for any $n\in N$ there is $m\in M$ such that $m\leq n$; then we write $M<N$. Put 
$$M_1:=\{\omega: 0<\omega\leq1,\log\omega\in \Lip_1(\mathbb{R}), \mathcal{L}(\omega)>-\infty\},$$
$$M_2:=\{1\slash|f\bigr{|}_\mathbb{R}|: f\in Cart, |f|\geq1\text{ on }\mathbb{R}\}.$$
Koosis proved that $M_1<M_2$ and $M_2<M_1$ whence $M_j\subset BM$ implies $M_k\subset BM$ for any choice of $j,k=1,2$. Note that 
none of $M_1$, $M_2$ is contained in the other (see \cite{HMN}).

\subsection{On the proofs of the BM-theorem}
We conclude this section with a short discussion of the original proof in \cite{BM} and \cite{M} (see also \cite[Part 2, Ch.3]{HJ}) and with some information on the subsequent proofs.

\subsubsection{}
The proofs in \cite{BM} and \cite{M} result in a very general assertion implying the statements in our Subsections \ref{Bor1} and \ref{Bor2}. This assertion invloves the integral $E(k):=\int_\mathbb{R}\int_\mathbb{R}\bigl{(}\frac{k(x)-k(y)}{x-y}\bigr{)}^2dxdy$ where $k(x)=\frac{\Omega(x)}{x}$ (we assume $\Omega\equiv 0$ in a vicinity of the origin). 
The convergence of $E(k)$ means that "the energy" $\int_{\mathbb{C}_+}|\bigtriangledown u|^2dxdy$ of the harmonic extension $u$ of $k$ to $\mathbb{C}_+$ (by the Poisson integral) is finite. It is shown in \cite{BM} that the estimates $E(k)<+\infty$ and $\mathcal{L}(\omega)>-\infty$ imply the existence (for a given $\sigma>0$) of a non-zero $f\in L^2$ with $\spec f\subset[-\sigma,\sigma]$ such that $\int_\mathbb{R}\bigl{(}\frac{|f(x)|}{\omega(x)}\bigr{)}^2dx<+\infty$ or $|f|\leq\omega_\varepsilon$ where
$\omega_\varepsilon$ is a regularization of $\omega$ (see \cite{BM,M,HJ} and a detailed discussion in \cite{HMN}) .

\subsubsection{}
The assertions in \ref{41} and \ref{42} were repeatedly reproved many times. These proofs used various approaches and techniques interesting in their own right (see \cite{M, Koo1, Koo2, Koo3, HJ, HMN, H} and literature therein). The next subsection sketches a real variable approach proposed in \cite{HM1, HM2} where it led to simple proofs of some particular cases of the BM-theorem. A complete proof based on this approach and on a deep Nazarov's theorem on the Hilbert transform of Lipschitz functions (see \cite{HMN}) is also described briefly.

\subsubsection{}
Let $f$ be a function in $L^1((1+x^2)^{-1}dx)$. Put
$$\tilde{f}(x):=\frac{1}{\pi}p.v.\int_\mathbb{R}f(t)\biggl{(}\frac{1}{x-t}+\frac{t}{1+t^2}\biggr{)}dt$$
(the principal value of the integral exists and is finite for almost all $x\in\mathbb{R}$). The function $\tilde{f}$ is called {\it the Hilbert transform } of $f$. Return to the majorant $\omega:\mathbb{R}\rightarrow(0,1]$ with $\Omega:=|\log\omega|$ in $L^1((1+x^2)^{-1}dx)$. The following theorem is a corollary of a theorem by Dyakonov \cite{D} on the moduli of functions from the model spaces (see also Section 2 below):
\begin{theorem}\textnormal{(}\cite[Section 1.14]{HMN}\textnormal{)}
If $\Omega\in \Lip_1(\mathbb{R})$ and $\lim_{|x|\rightarrow\infty}\widetilde{\Omega}'(x)=0$, then $\omega\in BM$.
\label{lipth}
\end{theorem}
This condition sufficient for the inclusion $\omega\in BM$ is remotely similar to the $BM$-theorem, but is much less explicit being stated in terms of $\widetilde{\Omega}$, not of $\Omega$ itself. Recall that the Hilbert transform of  a $\Lip_1$-function is not bound to be $\Lip_1$, it may be not uniformly continuous and even worse (see \cite{Bl1}). The last theorem immediately implies the following corollary:
$$\Omega\in \Lip_\alpha(\mathbb{R})\quad\text{ and }\quad0<\alpha<1\quad\Rightarrow\quad\omega\in BM,$$
which is much weaker than the BM-theorem. But Theorem \ref{lipth} does imply the BM-theorem in its complete form due to the following deep result by F. Nazarov (\cite[Theorem 2]{HMN}):
\begin{theorem}
If $\Omega\in L^1((1+x^2)^{-1}dx)$, $\Omega\geq0$, $\varepsilon>0$, then there exists an 
$\Omega_1\in L^1((1+x^2)^{-1}dx)$ such that $\|\widetilde{\Omega}'\|_\infty<\varepsilon$, $\Omega_1\geq\Omega$.
\label{Naz}
\end{theorem}

\subsubsection{On applications of the BM multiplier theorem. \label{secondBM}} We have already mentioned connections of the BM-theorem with the Uncertainty Principle (see 1.2.1). Applications to various problems of analysis are discussed in \cite{Koo2} (the title of Chapter 10 is "Why we want to have multiplier theorems"). Among the themes related to the BM-theorem are normal families of trigonometrical sums, weighted estimates of the Hilbert transform, weighted polynomial approximation and approximation by functions with bounded spectra (\cite[pp. 468--469]{Koo2}, \cite[p.174]{Koo3}, \cite[Ch. 3]{deBr}, \cite[sect. 3.2]{HMN}).

We turn now to the so-called "Second BM-theorem" on complete families of exponentials $E_\lambda$, $E_\lambda(x):=e^{i\lambda x}$, $x\in\mathbb{R}$. Given a discrete set $\Lambda$ of real numbers and a positive number $R$ we ask whether the family $E(\Lambda):=\{E_\lambda\}_{\lambda\in\Lambda}$ is complete in $L^2(-R,R)$. This natural question is, however, too precise to admit a clear and efficient answer. The following problem looks more realistic: given $\Lambda$ find 
$$R(\Lambda):=\sup\{r\geq0: E(\Lambda)\text{ is complete in } L^2(-r,r)\}.$$
The famous "Second BM-theorem" yields an explicit expression of $R(\Lambda)$ as certain density $D(\Lambda)$ of $\Lambda$ at infinity. We won't reproduce here the definition of $D(\Lambda)$ and only hint at the connection of the "First" BM-theorem (as stated in \ref{42}) with the "Second". An obvious duality argument reduces the equality $R(\Lambda)=D(\Lambda)$ to the following uniqueness problem for the Paley--Wiener class $\mathcal{PW}_r$ (see 1.1): is it true that
\begin{enumerate}
\begin{item}
$0<r<D(\Lambda) \Rightarrow \Lambda$ is a uniqueness set for $\mathcal{PW}_r$,
\end{item}
\begin{item}
$r>D(\Lambda)\Rightarrow$ there exists a non-zero $f\in\mathcal{PW}_r$ vanishing on $\Lambda.$
\end{item}
\end{enumerate}
Part $(i)$ is relatively easy. Part $(ii)$ is quite hard. The multiplier theorem allows to replace $\mathcal{PW}_r$ in $(ii)$ by a larger class $\{f\in Cart: |f(z)|=O(e^{r|z|}), |z|\rightarrow\infty\}$, thus simplifying the problem. For the proofs of the "Second BM-theorem" and its generalizations see \cite{Koo2, Koo3, HJ, MP1}; these items contain a lot of further references.

\section{On the moduli of functions in the de Branges spaces\label{s2}}

\subsection{Setting of the problems} We turn now to moduli majorants for the de Branges spaces of entire functions. The Paley--Wiener spaces $\mathcal{PW}_\sigma$ (see 1.1) are a particular case. Given a non-negative function $\omega$ on $\mathbb{R}$ we are still interested in the existence of a non-zero element $f$ of a given space of entire functions satisfying the estimate $|f|\leq\omega$. We will also consider the possibility of two-sided estimates $|f|\asymp\omega$ (see Subsection \ref{two}). These two themes can be combined as we shall see later.

\subsubsection{De Branges spaces.} For a function $f:\mathbb{C}\rightarrow\mathbb{C}$ put $f^*(z)=\overline{f(\overline{z})}$, $z\in\mathbb{C}$.

The Paley--Wiener spaces $\mathcal{PW}_\sigma$ obviously satisfy the following two axioms:

\begin{enumerate}
\begin{item}
$f^*\in\mathcal{PW}_\sigma$ whenever $f\in\mathcal{PW}_\sigma$, and $\|f\|_{L^2}=\|f^*\|_{L^2}$.
\end{item}
\begin{item}
If $f\in\mathcal{PW}_\sigma$, $\lambda\in\mathbb{C}$, and $f(\lambda)=0$, then the function $z\mapsto f(z)\dfrac{z-\bar{\lambda}}{z-\lambda}$ is in $\mathcal{PW}_\sigma$ and has the same norm as $f$.
\end{item}
\end{enumerate}

The remarkable fact is that all  Hilbert spaces of entire functions with a reproducing kernel satisfying these two axioms can be described explicitly. This is one of possible definitions of the de Branges spaces. Now we remind another (equivalent) definition of these spaces. 

We say that an entire function $E$ belongs to the Hermite--Biehler class if it has no real zeros and satisfies 
$$|E(z)| > |E(\bar{z})|$$
for $z$ in the upper half-plane $\mathbb{C}_+$. For a given function $E$ in the Hermite--Biehler class we let
$\he$ be the Hilbert space consisting of all entire functions $f$ such that both $f\slash E$ and
$f^*\slash E$ belong to $H^2$, where $H^2$ is the Hardy space in $\mathbb{C}_+$ (see Subsection \ref{hardy}),
identified in the usual way with a subspace of $L^2(\mathbb{R})$. We set
$$\|f\|^2_{\he}=\int_{\mathbb{R}}\frac{|f(x)|^2}{|E(x)|^2}dx.$$
We arrive at the Paley--Wiener space $\mathcal{PW}_\pi$ by setting $E(z)=e^{-i\pi z}$. 
The reproducing kernel corresponding to the point $w\in\mathbb{C}$ 
is given by 
\begin{equation}
k_w(z)=\frac{\overline{E(w)}E(z)-\overline{E^*(w)}E^*(z)}{2\pi i(\overline{w}-z)}.
\label{repr}
\end{equation}

\subsubsection{Model subspaces $K_\Theta$.}
Sometimes it is natural to consider $\he$ as a subspace of $H^2$. More precisely, the space $\dfrac{1}{E}\he$ is a shift {\it coinvariant 
} subspace of $H^2$ and, hence, is of the form $H^2\cap \Theta \overline{H^2}$, where $\Theta$ is an inner function in the upper half-plane $\mathbb{C}_+$ (see Subsection \ref{122}) . Moreover,
\begin{equation}
\Theta(z)=\frac{E^*(z)}{E(z)},\quad z\in\mathbb{C}_+.
\label{tE}
\end{equation}
This formula defines $\Theta$ in the lower half-plane as a {\it meromorphic function}. On the other hand it is well known that (see \cite[Lemma 2.1]{HM1}) {\it any} meromorphic inner function can be represented in such a way. The function $E$ from \eqref{tE} is unique up to an entire factor $S$ which is real on the real line and has no zeros. This function causes no problem because $\mathcal{H}(SE)=S\he$ for any de Branges space $\he$ and any such $S$.

Put $K_\Theta = H^2\cap \Theta \overline{H^2}$. $K_\Theta$ is called {\it a model subspace} of $H^2$ corresponding to $\Theta$ (see, e.g., \cite{HM1}). As it was previously mentioned moduli of functions from $H^2$ can be described explicitly, 
$$\omega\in |H^2(\mathbb{R})|\Leftrightarrow \omega\in L^2(\mathbb{R}): \omega\geq0, \mathcal{L}(\omega)=\int_\mathbb{R}\frac{\log\omega(x)}{1+x^2}dx>-\infty.$$
This makes the class $K_\Theta=\dfrac{1}{E}\he$ more suitable for our purposes. 

For any meromorphic inner function $\Theta$ there exists a real continuous and increasing function $\varphi(=\phi_\Theta)$ on $\mathbb{R}$ such that 
$$\Theta(x)=e^{i\varphi_\Theta(x)}=:e^{-2i\arg E(x)},\quad x\in\mathbb{R}.$$
This function is unique up to an additive constant $2\pi k, k\in\mathbb{Z}$. Almost all results will be expressed in terms of $\varphi$
(called {\it the phase function} of $\Theta$). 

\subsection{First results}
The next (rather simple) result was a starting point to the extensive investigation of moduli of functions from $K_\Theta$, \cite{Bl2, BH, BBH, HM1, HM2, D}. It describes (non-explicitly) {\it} all non-trivial moduli of functions from $K_\Theta$.

\begin{theorem}
Let $\Theta$ be a meromrophic inner function and  $f$ be a non-zero function from $K_\Theta$. Then there exist an inner function $I$ and a non-decreasing integer-valued function $k$ on $\mathbb{R}$ such that
\begin{equation}
\phi_\Theta=2\widetilde{\log|f|}+\arg I+2\pi k.
\label{feq}
\end{equation}
Moreover, if an inner function $I$, a non-decreasing integer-valued function $k$ and a non-negative function $\omega\in L^2(\mathbb{R})$ are such that 
$$\phi_\Theta=2\widetilde{\log\omega}+\arg I+2\pi k,$$
then there exists an $f\in K_\Theta$ with $|f|=\omega$. 
\end{theorem}

For reader's convinience we give a sketch of the proof.
\begin{proof}
We will use only the existence of the involution $f\mapsto \Theta\overline{f}$ of $K_\Theta$ (this corresponds to axiom $(i)$ in the  de Branges spaces setting). From $f\in H^2, \Theta\overline{f}\in H^2$, we conclude that $|f|^2\Theta\in H^1$. So, $|f|^2\Theta=OI$, where $O$ and $I$ are respectively the outer and inner factors. Taking into account an expression for the outer function
$$O=O_{|f|^2}=e^{\log|f|^2+i\widetilde{\log|f|^2}}$$ we get the first statement of the theorem. If we reverse steps of this proof, we get the second statement.
\end{proof}

If we are looking for a non-zero function $f\in K_\Theta$ with $|f|\leq \omega$, we have to find {\it a multipier} $m$ which satisfies some equation. The next result was obtained in \cite{HM1}. 
\begin{theorem}
Let $\omega$ be a non-negative function with $\mathcal{L}(\omega)>-\infty$. There exists a non-zero function $f\in K_\Theta$ such that
$|f|\leq \omega$ if and only if there exists a nonnegatvie function $m\in L^\infty(\mathbb{R})$, and an inner function $I$ such that  $m\omega\in L^2(\mathbb{R})$, and 
\begin{equation}
\phi_\Theta-2\widetilde{\log\omega}=2\widetilde{\log m}+\arg I +2\pi k,
\label{main}
\end{equation}
where $k$ is an integer-valued function on $\mathbb{R}$. Moreover, $m\omega\in |K_\Theta|$.
\end{theorem}
So, we have to represent the given function $\phi_\Theta-2\widetilde{\log\omega}$ as a sum $2\widetilde{\log m}+\arg I +2\pi k$.
It turns out that we can always assume $I$ is a unimodular constant by the following result \cite[Theorem 2.1]{BH}
\begin{theorem}
Let $I$ be an arbitrary inner function. Then there exists a nonnegative function $m\in L^\infty(\mathbb{R})\cap L^2(\mathbb{R})$ with $\mathcal{L}(m)>-\infty$, $\gamma\in\mathbb{R}$, and an integer-valued function $k$ such that
$$\arg I =2\widetilde{\log m}+2\pi k + \gamma\quad \text{ a.e. on }\mathbb{R}.$$
\label{HMTh}
\end{theorem}
The summand $k$ in the representations \eqref{main}, \eqref{feq} corresponds to the {\it real} zeros of $f$ (if $k$ has a jump at some point, then $f$ vanishes at this point and vice versa). On the other hand, summand $\arg I$ may be helpful if we are looking for $f$ whose modulus is comparable to $\omega$.

\subsection{Regular behaviour of the phase function.}

\subsubsection{}
We introduce a wide class of functions which can be represented as in the right-hand side of \eqref{main}. Let $\{d_n\}$ be an increasing sequence of real numbers. We assume that either $n\in\mathbb{Z}$ and $\lim_{|n|\rightarrow\infty}|d_n|=\infty$ or $n\in\mathbb{N}$ and $\lim_{n\rightarrow\infty}d_n=\infty$; in the latter case we set $d_0=-\infty$. Let $I_n=(d_n,d_{n+1})$.

\begin{definition}
An absolutely continuous function $f$ on $\mathbb{R}$ is said to be {\it mainly increasing} if there exists an increasing sequence $\{d_n\}$ as above such that $f(d_{n+1})-f(d_n)\asymp 1$, $n\in\mathbb{Z}$ ($n\in\mathbb{N}$), and there is a constant $C>0$ such that for any $n$
\begin{equation}
\sup_{s,t\in I_n}(f(s)-f(t))\leq C,\quad\frac{1}{|I_n|}\int_{I_n}|f'(x)-f'(t)|dt\leq C, \quad x\in I_n.
\label{maininceq}
\end{equation}
In the case of one-sided sequences $\{d_n\}$, we assume that $f$ is a Lipschitz function on $(-\infty,d_1)$.
\label{maininc}
\end{definition}
The integral condition is implied by $\sup_{s,t\in I_n}(f'(s)-f'(t))<\infty$. If for example $f'\asymp 1$, then $f$ is, obviously, mainly increasing.
\begin{theorem}
Let $f$ be a mainly increasing function. Then $f$ admits the representation $f=2\widetilde{\log m}+2\pi k + \gamma \text{ a.e. on }\mathbb{R}$, where $m\in L^{\infty}(\mathbb{R})\cap L^2(\mathbb{R})$, $\mathcal{L}(m)>-\infty$, $\gamma\in\mathbb{R}$ and $k$ is an integer-valued function.
\label{HMTh2}
\end{theorem}
This theorem was proved in \cite[Theorem 1.2]{HM2} under an additional restriction on $\{d_n\}$ and, finally, in \cite[Theorem 1.4]{BH}. Of course it is easy to choose $k$ so that $f-k$ will be bounded but the main difficulty is to make the Hilbert transform of this function bounded. 

Considering Theorems \ref{HMTh} and \ref{HMTh2} we immediately get
\begin{theorem} Let $\omega$ be a non-negative function with $\mathcal{L}(\omega)>-\infty$.
If $\varphi_\Theta-2\widetilde{\log \omega}$ is mainly increasing, then there exists  a nonzero $f\in K_\Theta$ such that $|f|\leq\omega$.
\end{theorem}

This theorem is a source of some results for the case when the  phase function $\varphi$ has a regular growth.
The main application is of course Theorem \ref{lipth}.

\subsubsection{}
Suppose $B$ be a Blaschke product 
$$B(z)=\prod_ke^{i\alpha_n}\biggl{(}\frac{z-z_k}{z-\overline{z_k}}\biggr{)}^{m_k},\quad \alpha_n\in\mathbb{R}$$
(here $\alpha_n\in\mathbb{R}$ and the factors $e^{i\alpha_n}$ ensure the convergence of the product). Then the
subspace $K_B$ admits a simple geometrical description: it coincides with the closed
linear span in $L^2(\mathbb{R})$ of the fractions $\dfrac{1}{(z-\overline{z_k})^n}$, $1\leq n\leq m_k$. The behavior of the phase function $\varphi_B$ depends on the properties of the sequence $\{z_k\}$. Here we give two examples.
\begin{theorem} \textnormal{(}\cite[Theorem 1.4]{HM2}\textnormal{)}
Let $B$ be a Blaschke product with almost uniformly distributed zeros in a horizontal strip, i.e., $ 0<c<\Im z_k<C$ and
there exist numbers $L,K>0$ such that for any $a\in\mathbb{R}$ the rectangle $[a,a+L]\times[c,C]$ contains at least one and not more than $K$ zeros. If $\widetilde{\log\omega}$ is Lipschitz with a sufficiently small Lipschitz constant, then there exists a non-zero $f\in K_B$ such that $|f|\leq \omega$.
\end{theorem}
In this case the phase function $\varphi$ satisfies $\varphi'\asymp1$.

  Let $B_\alpha$ be a Blaschke product with "horizontal" zeros $\{|k|^\alpha\sgn k +i\}_{k\in\mathbb{Z}}$, $1\slash2<\alpha\leq1$ (the condition $1\slash2<\alpha$ is necessary for the convergence of the Blaschke product). If $\alpha = 1$, then it is easy to see that $\sin(\pi(z+i))K_{B_1}=\pw$ and we arrive at the classical case. On the other hand, if $\alpha<1$ then it is easy to show that $\varphi'_B\asymp (1+|x|)^{\alpha^{-1}-1}$ and so we have a superlinear grows of the phase function. 

\begin{theorem}\textnormal{(}\cite[Theorem 1.10]{HM2}\textnormal{)}
 Suppose $\mathcal{L}(\omega)>-\infty$, $\widetilde{\log\omega}\in C^1(\mathbb{R})$ and 
$$-\frac{\pi}{\alpha}<\lim\inf_{|x|\rightarrow\infty}\frac{-(\widetilde{\log\omega})'(x)}{|x|^{\alpha^{-1}-1}}\leq 
\lim\sup_{|x|\rightarrow\infty}\frac{-(\widetilde{\log\omega})'(x)}{|x|^{\alpha^{-1}-1}}<\infty.$$ 
Moreover, suppose that the continuity modulus $\lambda_t$ of $(\widetilde{\log\omega})'$ on $\mathbb{R}\setminus(-t,t)$ satisfies $\lambda_t(t^{1-\alpha^{-1}})\in L^\infty(\mathbb{R}_+)$. Then there exists a non-zero $f\in K_{B_\alpha}$ such that $|f|\leq \omega$.
\end{theorem}
Some generalization of this result can be found in \cite{Bl2} (Theorem 5). The main disadvantage of these results is that we impose some conditions on the Hilbert transform of $\log\omega$ and not $\log\omega$ itself. But as in the classical case we can use the brilliant Theorem \ref{Naz} by Fedor Nazarov  to get rid of this problem. This was done in \cite{Bl2} (Theorems 9, 10). To avoid inessential definitions here we cite only one corollary   
of these results.
\begin{corollary}\textnormal{(}\cite[Corollary 10.1]{Bl2}\textnormal{)}
Let $B$ be a Blaschke product with zeros $\{|k|^\alpha\sgn k +i|x|^\beta\}_{k\neq0}$, where $1\slash2<\alpha<1$ and $\alpha-1<\beta<0$. Suppose that $\mathcal{L}(\omega)>-\infty$ and $|(\log\omega)'(x)|\leq M(1+|x|)^{\alpha^{-1}-1}$ for  some $M$ and all $x$'s. Then there exists a non-zero $f\in K_B$ such that $|f|\leq\omega$
\end{corollary}
Some results about the sharpness of these conditions can be found in \cite{Bl3}.

The situation when the argument grows sublinearly is much simpler.  For a regular and sublinear growth of argument (e.g., $\varphi'_\Theta(x) \asymp (1+ |x|)^{\beta- 1}$, $\beta \in (0,1)$) we have $1\in\he$ and $1\slash E\in K_\Theta$.  
The majorant $1\slash E$ is {\it minimal}, i.e., any non-zero function $f\in K_\Theta$ with $|f|\leq 1\slash |E|$ 
is comparable to $1\slash E$. Indeed, if an entire function $F$ of zero exponential type is bounded on $\mathbb{R}$,
 then $F\equiv const$. We refer to Section 3 of \cite{HM1} for the details. The conditions on the zeros of $E$ ensuring the enclusion $1 \in \mathcal{H}(E)$ were studied in \cite{Bar1,W,KW}.

\subsection{Irregular behaviour of the phase function. Zeros in the right half-plane} 
\subsubsection{}
A meromorphic inner function $\Theta$ is of the form
$\Theta(z)=e^{iaz}B(z)$, $z\in\mathbb{C}_+$, where $a\geq0$ and $B$ is a Blaschke product with zeros tending to infinity. It is well known that
\begin{equation}
\phi_\Theta'(x)=|\Theta'(x)|=a+2\sum_n\frac{m_n\Im z_n}{|x-z_n|^2},\quad x\in\mathbb{R},
\label{arg}
\end{equation}
where $B$ is a Blaschke product with zeros $z_n$ of multiplicities $m_n$.There are at least two reasons for bad behaviour of the right hand side of \eqref{arg}:
\begin{itemize}
\begin{item}
If there are big gaps in the sequence $\{\Re z_n\}$, then $\varphi_\Theta$ grows slowly on large  intervals and functions from $K_\Theta$ have more or less prescribed behavior on such intervals.
This corresponds to the situation when intervals $I_n$ from Definition \ref{maininc} are very long.
\end{item}
\begin{item}
If the zeros $z_n$ are very close to the real line, then $\varphi$ "almost jumps" at the points $\Re z_n$ and there is no hope for the cancellation in the integral inequalities in \eqref{maininceq}.
\end{item}
\end{itemize}
The model example of the first situation is the absence of zeros in the left half-plane $\{\Re z_n \leq 0\}$. This case was studied in  \cite{BBH}. The second situation will be discussed in Subsection \ref{tang}.

\subsubsection{}
Let $B^+_\alpha$ be the Blaschke product with zeros $z_n=n^\alpha+i$, $n\in\mathbb{N}$, $\alpha > 1\slash2$.

The prime example is of course $\alpha = 1$. It is obvious that $K_{B^+_1}\subset K_{B_1}$, but on the other hand there exists a nontrivial $f\in K_{B_1}$ such that there is no nontrivial $f^+\in K_{B^+_1}$
with $|f^+(x)|\leq |f(x)|$. This happens because functions from $K_{B^+_1}$ are much more regular than elements of $K_{B_1}$. The following theorem similar to the classical Paley result for $\pw$ space (or equivalently for $K_{B_1}$) was obtained in \cite{BBH}.
\begin{theorem}\textnormal{(}\cite[Theorem 1.1]{BBH}\textnormal{)} If $f\in K_{B^+_1}\setminus\{0\}$, then
\begin{equation}
\int_{0}^{+\infty}\frac{\log |f(t)|}{1+t^{3\slash2}}dt>-\infty.
\label{log32}
\end{equation}
Moreover, if $\omega$ is an even positive non-increasing function on the positive semiaxis and the integral \eqref{log32} converges, then
there exists a nontrivial $f\in K_{B^+_1}$ with $|f(x)|\leq \omega(x)$, $x>0$.
\end{theorem}
It is interesting that there is more freedom in the behaviour of the elements of $K_{B^+_1}$ along the negative semiaxis. More precisely, there exists  a nonzero function $f\in K_{B^+_1}$
such that $\log|f(x)|\leq -|x|^{1\slash2}$, $x<0$, whereas this is impossible for $x>0$. Moreover, this result is sharp.
\begin{theorem} \textnormal{(}\cite[Theorem 1.2]{BBH}\textnormal{)}
For any $A>0$ there exists a nonzero function $f\in K_{B^+_1}$
such that $\log|f(x)|\leq -A|x|^{1\slash2}$, $x<0$. At the same time, if $1+|x|^{1\slash2}=o(-\log|f(x)|)$, $x\rightarrow-\infty$, then $f\equiv 0$.
\label{neg}
\end{theorem}
These two results have a certain informal explanation. It is not difficult to show that elements of $K_{B^+_1}$ are bounded (and analytic) not only in $\mathbb{C}_+$ but also in the domain
$$\Delta = \mathbb{C}\setminus \{z: \Re z\geq 0, -2\leq\Im z\leq 0\}.$$
Let $\nu$ be the conformal mapping of the
upper halfplane $\mathbb{C}_+$  onto the domain $\Delta$ such that $\nu(0)=0$, $\nu(\infty)=\infty$. By the
Christoffel--Schwarz formula, $\nu$ is of the form
\begin{equation}
\nu(z)=a_1+a_2\int_{z_0}^z\zeta^{1\slash2}(\zeta-a)^{1\slash2}d\zeta,
\label{nu}
\end{equation}
where $a_1\in\mathbb{C}$, $a_2>0$, $z_0\in\mathbb{C}_+$ and $a > 0$. Function $\nu$ is very close to $a_2z^2\slash2$. It is clear that
$$\nu(z)=a_2z^2\slash2 +o(z^2),\quad \nu'(z)=a_2z+o(z),\quad z\rightarrow\infty.$$ 
So we can think that $K_{B^+_1}$ "becomes" $K_{B_1}$ after the substitution $z\mapsto z^{1\slash2}$. The Poisson measure $\dfrac{dt}{1+t^2}$ becomes $\dfrac{ds}{1+s^{3\slash2}}$, $t=s^{1\slash2}$.
At the same time the behaviour of functions from $K_{B^+_1}$ on the negative semiaxis is similar to the behaviour of functions from $K_{B_1}$ on {\it the imaginary} axis. The last question is an essence related to the $\mathcal{PW_\pi}$ space. This argument explains Theorem \ref{neg}. 

\subsubsection{}
Now we turn to spaces $K_{B^+_\alpha}$ when $\alpha\neq 1$. We will study possible (global) rates of decay of functions from $K_{B^+_\alpha}$. Put
$$A(\alpha) = \sup\{s: \text{ there exists a nonzero } f\in K_{B^+_\alpha} \text{ such that } \log|f(x)|\leq -|x|^s \},$$
$$A_+(\alpha) = \sup\{s: \text{ there exists a nonzero } f\in K_{B^+_\alpha} \text{ such that } \log|f(x)|\leq -x^s, x>0 \}.$$
We define $A_{-}(\alpha)$ analogously. Numbers $A_{\pm}(\alpha)$ are rough characteristics of respective spaces. Their behaviour is, nevertheless, complicated.
\begin{theorem}\textnormal{(}\cite[Theorem 1.4]{BBH}\textnormal{)}
$$A(\alpha)=A_{+}(\alpha)=
\begin{cases}
1\slash\alpha,\qquad \alpha >2,\\
1\slash2,\qquad 2\slash3\leq\alpha\leq 2,\\
-1+1\slash\alpha,\qquad 1\slash2<\alpha<2\slash3,
\end{cases}
A_{-}(\alpha)=
\begin{cases}
1\slash\alpha,\qquad \alpha >2,\\
1\slash2,\qquad 1\leq\alpha\leq 2,\\
-1,\qquad 1\slash2<\alpha<1.
\end{cases}
$$
\label{A}
\end{theorem}
The main ingredients of the proof are the conformal mapping $\nu$ (see \eqref{nu}) and an estimate of functions from $K_{B^+_\alpha}$ in the domain $\Delta$.
\begin{lemma}\textnormal{(}\cite[Lemma 3.2]{BBH}\textnormal{)} Let $f$ be in $K_{B^+_\alpha}$. If $\alpha\geq 1$, then $f$ is bounded in $\Delta$. If $1\slash2<\alpha<1$, then
$$\log|f(z)|\lesssim (1+|z|)^{-1+1\slash\alpha}.$$
\end{lemma}

We shortly describe one of possible proofs.  Let us return to the corresponding de Branges space $\he$. The function $Ef$ is in $\he$ and
$$|E(z)f(z)|^2\leq \|Ef\|^2_{\he}\cdot\|k_z(\cdot)\|_{\he}^2\leq C_f \frac{|E(z)|^2-|E^*(z)|^2}{4\pi \Im z},\quad z\in\mathbb{C}.$$
 Function $E$ is a canonical product with respect to the zero sequence $k^\alpha-i$, $k\in\mathbb{N}$ and satisfies $\log|E(z)|\asymp  1+|z|^{1\slash\alpha}$ outside an exceptional set. 
On the other hand $|E^*(z)|=|E(z-2i)|$. This gives the required estimate.

It is interesting to note that Theorem \ref{A} is closely related to weighted approximation (see Remark in \S5 of \cite{BBH}).

Analogous results are obtained for the Blaschke products with two-sided zeros
having different power growth in the positive and negative directions (see \cite[Theorem
5.6]{BBH}).

\subsection{Irregular behaviour of the phase function. Tangential zeros. \label{tang}} 

In this subsection we consider the situation when zeros of the Blaschke product approach  the real axis tangentially.
Let $B$ be the Blaschke product with zeros $z_n=n+iy_n$, where $0<y_n\leq1$. If $y_n$ tend to zero not too rapidly, then there is no
qualitative difference between the classes $K_B$ and $K_{B_1}$.
\begin{theorem}\textnormal{(}\cite[Corollary 4]{Bl3} and \cite[Theorem 1.5(1)]{BBH}\textnormal{)}
Let a sequence $\{y_n\}_{n\in\mathbb{Z}}$ be even and nonincreasing for $n\geq 0$. If
$$\sum_n\frac{\log y_n}{1+n^2}>-\infty,$$
then for any even positive function $\omega$ with $\mathcal{L}(\omega)>-\infty$, there exists a nonzero $f\in K_B$ such that $|f|\leq\omega$.
\label{logy1}
\end{theorem}

The next result gives other conditions on $y_n$ and is closer to the original Beurling--Malliavin statement.
\begin{theorem}\textnormal{(}\cite[Corollary 4]{Bl3}\textnormal{)}
Let a bounded positive sequence $\{y_n\}_{n\in\mathbb{Z}}$ be such that $y_n\asymp y_{n+1}$. If
$$\sum_n\frac{\log y_n}{1+n^2}>-\infty,$$
then for any positive function $\omega$ such that $\mathcal{L}(\omega)>-\infty$ and $\log \omega$ is Lipschitz there exists a nonzero $f\in K_B$ such that $|f|\leq\omega$.
\end{theorem}

The proofs of these results include ideas from the classical case and an accurate estimate of the difference of phase functions of the Blaschke products $B(=B_{\{y_n\}})$ and $B_1$.

 On the other hand if $y_n$ are extremely small, then any nonzero function from $K_B$ decays as a power at infinity. Before stating the result we give an interepretation of this phenomenon.
Let us return to the de Branges setting. Supppose all zeros of the generating function $E$ are extremely close to the real line. Let $F$ be an arbitrary nonzero element of $\he$. The inclusion $\dfrac{F}{E}\in L^2(\mathbb{R})$ implies that $F$ has a zero near every zero of $E$ (may be excluding a finite number of zeros). In addition $F$ grows along the imaginary axis not faster than $E$. So, $F$ has at most a finite number of extra zeros and, hence, $\dfrac{F}{E}$ tends to zero not faster than a power. This property is called {\it strong localization } property of the space $\he$ (see the discussion in the Introduction of \cite{ABB}). 

As it was shown in \cite{ABB} the {\it strong localization } property is equivalent to the completenss of polynomials in the corresponding weighted space of sequences.

\begin{theorem}\textnormal{(}\cite[Theorem 1.5(2)]{BBH}\textnormal{)}
Let $y:\mathbb{R}\rightarrow (0,\infty)$ be an even function nonincreasing on $[0,\infty)$ and such that
$y(n)=y_n$, $n\in\mathbb{Z}$. If the function $-\log y(e^x)$ is convex on $\mathbb{R}$ and $\sum_n\frac{\log y_n}{1+n^2}=-\infty$, then
any function $f\in K_B$ satisfies $\lim\sup_{x\rightarrow\infty}|f(x)x^N|>0$ for some $N>0$.
\label{logy2}
\end{theorem}

In \cite{BBH} a result similar to Theorems \ref{logy1} and \ref{logy2} is proved for the case when zeros of $B$ are both one-sided and tangential ($z_k=k^\alpha+iy_k$, $k\in\mathbb{N}$).

\subsection{Two sided estimates. Atomization procedure. \label{two}} 

The approximation of subharmonic functions in the complex plane by logarithms of moduli of entire functions was a longstanding problem
in complex analysis. In \cite{Az} it was proved that if a subharmonic $u$ satisfies $u(z)\lesssim 1+|z|^{\rho+\varepsilon}$ for any $\varepsilon>0$, then there exists an entire function $f$ such that 
$$u(z)-\log|f(z)|=o(|z|^\rho),\quad z\rightarrow\infty,\quad z\notin E_f,$$
where $E_f$ is an appropriate small exceptional set. This result was modified by many authors. A breakthrough has been achieved in \cite{Yul}, where the righthand side was reduced to $O(\log|z|)$. Finally in \cite{LM} 
the best possible constant in $O(\log|z|)$ has been found. Moreover, if $u$ has some extra regularity, then there exists an entire function $f$ whose modulus is comparable to $e^u$,
 $u(z)-\log|f(z)|=O(1)$, \cite[Theorem 3]{LM}.

The construction of the approximating entire function is carried out by an "atomization" of the Riesz measure $\mu$ of $u$. For a given $\mu$ we have to find an atomic measure $\mu_a$ such that the logarithmic potential of $\mu-\mu_a$ is bounded (outside an exceptional set). One of the ideas of the atomization is to choose atoms (zeros of $f$) so that the first moments of the measure $\mu-\mu_a$ with respect to some domains vanish.

We have two special features of our problem:
\begin{itemize}
\begin{item}
We need approximate on the real line only.
\end{item}
\begin{item}
There is an additional restriction $f\in\he$ on the approximating entire function.
\end{item}
\end{itemize}

The second does not allow us to directly apply the results of \cite{Yul} and \cite{LM}. If we want to use Theorem \ref{main} for the construction of $f$ whose modulus is comparable to a given positive $\omega$, we have to put $k\equiv const$ and find an inner function $I$ such that $\phi_\Theta-2\widetilde{\log\omega}-\arg I$ can be represented as the Hilbert transform of a bounded function. As we will see later an atomization procedure is hidden here.

\subsubsection{} To state the theorem we need two definitions.

\begin{definition} 
A partition of the real line into intervals $I_k=[d_k,d_{k+1}]$ (where $\{d_k\}$ is a strictly increasing two-sided sequence) is said to be uniformly short if 
\begin{equation}
\sup_{k\in\mathbb{Z}}\sum_{|k-l|>1}\frac{|I_l|^2}{\dist^2(I_k,I_l)}<\infty.
\label{unshort}
\end{equation}
\end{definition}
Here $|I_l|$ stands for the length of the interval $I_l$, and $\dist(I_k,I_l)$ is the distance between $I_k$ and $I_l$. If the sum converges for all $k$ (but may be not uniformly), then we have the usual definition
of {\it a system of short intervals} which often appears in the Beurling--Malliavin theory. 

\begin{definition}
An increasing function $\Phi$ on $\mathbb{R}$ is said to be regular if there exists a two sided sequence $\{d_k\}$ with $\Phi(d_k)=2\pi k$, $k\in\mathbb{Z}$, such that the partition $I_k=[d_k, d_{k+1}]$ is uniformly short and 
$\sup_{|\Phi(x)-\Phi(y)|<1}\dfrac{\Phi'(x)}{\Phi'(y)}<\infty$.
\end{definition}
The last condition holds if $\Phi$ is mainly increasing.
\begin{theorem}\textnormal{(}\cite[Theorem 2.4]{Bl4}\textnormal{)}
Let a positive $\omega\in L^2(\mathbb{R})$ be such that $\mathcal{L}(\omega)>-\infty$. If $\varphi_\Theta-2\widetilde{\log\omega}$ is regular, then
there exists a function $f\in K_\Theta$ such that $|f|\asymp \omega$.
\label{twoth}
\end{theorem}
If, for example, $f$ satisfies $f'\asymp 1$, then $|I_l|\asymp1$ and, hence, $f$ is regular. So, we get an immediate corollary of this result for the Paley--Wiener space.
\begin{theorem}\textnormal{(}\cite[Theorem 2.6]{Bl4}\textnormal{)}
Let a positive $\omega\in L^2(\mathbb{R})$ be such that $\mathcal{L}(\omega)>-\infty$. If $-\infty<\inf_x (\widetilde{\log\omega})'(x)\leq\sup_x (\widetilde{\log\omega})'(x)<\pi$, then
there exists a function $f\in\pw$ such that $|f|\asymp \omega$.
\end{theorem}

\subsubsection{}
Now we briefly explain the idea of the proof of  Theorem \ref{twoth}. 
Put $\Phi:=\varphi_\Theta-2\widetilde{\log\omega}$. We have to find an inner $I$ 
such that $\Phi-\arg I=\widetilde{\log m}$, $\log m\in L^\infty$. Let us fix the
 sequence $\{d_k\}_{k\in\mathbb{Z}}$ so that $\Phi(d_k)=2\pi k$, $k\in\mathbb{Z}$,
and the corresponding partition $I_k=[d_k, d_{k+1}]$. Let $I$ be a Blaschke product
 whose zeros are $x_k+iy_k$ and $n_I$ be counting function of $\{x_k\}$. 
Put $y_k:=|I_k|=d_{k+1}-d_k$. We choose $x_k\in I_k$ so that $\int_{\mathbb{R}}(\Phi-2\pi n_I)=0$. 
These conditions come from the "atomization" procedure. 

 Note that the sequence ${x_k+iy_k}$ satisfies the Blaschke condition automatically. 
The function $n_I$ is close to $\arg I$.

The key argument is that function $\widetilde{\Phi -2\pi n_I}$ 
can be estimated outside a neighbourhood of $\{x_k\}$. We use the following well-known fact
\begin{proposition} Let $f$ be a bounded function on $I$ and $\int_If=0$. Then
$$\biggl{|}\int_I\frac{f(t)}{x-t}dt\biggr{|}\leq \frac{|I|^2\cdot\|f\|_\infty}{\dist^2(x,I)}.$$
\end{proposition}

Now we put
\begin{equation}
-\log m=\widetilde{\Phi-n_I}-\widetilde{(\arg I- n_I)}.
\label{m}
\end{equation}
It remains to verify $\log m\in L^\infty$ and to apply the Hilbert transfrom to both sides of \eqref{m}.
We refer to \cite[Section 4]{Bl4} for a detailed proof.

\section{Toeplitz kernel approach \label{s3}}

Toeplitz operators appear as a natural tool in the study of the Paley--Wiener space. 
For example in the important paper \cite{HNP} the description of {\it all}
 bases of exponentials was found using the {\it invertibility} properties of Toeplitz operators. 
Another advantage is that properties of Toeplitz operators are the unifying language of "the First BM-theorem" 
and "the Second BM-theorem" about completeness radius of exponentials. 
The multiplier Beurling--Malliavin theorem (and its generalizations) corresponds to the {\it injectivity} problem.

\subsection{Preliminary definitions} Let $u\in L^\infty(\mathbb{R})$. The {\it Toeplitz operator} 
with symbol $u$ is the map 
$$T_u: H^2\mapsto H^2,\quad f\mapsto P_{H^2}(uf),$$
where $P_{H^2}$ is the orthogonal projection of $L^2(\mathbb{R})$ onto $H^2$. 
From the definition it follows immediately that if $\Theta$ is an inner function, then 
$$\Ker(T_{\overline{\Theta}})=H^2\cap \Theta \overline{H^2}=K_\Theta.$$
Moreover, $\Lambda\subset \mathbb{C}_+$ is a uniqueness set for $K_\Theta$ if and only 
if the kernel of the operator $T_{B_\Lambda\overline{\Theta}}$ is trivial where $B_\Lambda$
 is the Blaschke product with the zero set $\Lambda$. Indeed,
if $T_{B_\Lambda\overline{\Theta}}(f)\neq0$, $f\in H^2$, then $B_\Lambda\overline{\Theta}\in \overline{H^2}$ 
and, hence $B_\Lambda f\in \Theta\overline{H^2}$. So, $B_\Lambda f$ is in $K_\Theta$ 
and vanishes at $\Lambda$. On the other hand, if $B_\Lambda f\in K_\Theta$, 
then $f\in \Ker(T_{B_\Lambda\overline{\Theta}})$.
The assumption $\Lambda\subset \mathbb{C}_+$ is not restricting for the de Branges 
spaces setting, since $\Lambda$ is a uniqueness set for $\he$ if and only 
if $(\Lambda\cap \overline{\mathbb{C}_+})\cup \overline{(\Lambda\cap\overline{\mathbb{C}_-})}$  is a uniqueness set.

The next standard object is the Smirnov-Nevanlinna class $\mathcal{N}^+=\mathcal{N}^+(\mathbb{C}_+)$. 
The elements of $\mathcal{N}^+$ are ratios $\dfrac{u}{v}$, where $u,v\in H^\infty(\mathbb{C}_+)$ 
and $v$ is an outer functuon. It is well known that functions from $\mathcal{N}^+$ have angular 
boundary values almost everywhere on the real line and $H^2=\mathcal{N}^+\cap L^2(\mathbb{R})$.
 We introduce the notations
$$N^+[u]=\{f\in\mathcal{N}^+\cap L^1_{loc}(\mathbb{R}): \overline{uf}\in \mathcal{N}^+\},\quad N^p[u]=N^+[u]\cap L^p(\mathbb{R}),\quad 0<p\leq\infty.$$
For $p=2$ we have $N^2[u]=\Ker(T_u)$. So, we can think that $N^{p \text{ or } +}[u]$ are generalizations of kernel of $T_u$ for other spaces. The Cartwright class in our notation is $\cup_{a>0}e^{-iaz}N^+[e^{-2iaz}]$.

\subsection{The BM-theorem in terms of Toeplitz kernels.} We will study the triviality of $N^{\cdot}[u]$ 
for the case when $u=e^{i\gamma}$ is a unimodular function. Now we can reformulate the Beurling--Malliavin multiplier theorem 
(Theorem \ref{mult}) in our language
\begin{theorem} Let $I$ be a meromorphic inner function. If $N^+[e^{-iaz}I]\neq 0$, $a>0$, then for any $\varepsilon >0$
$N^\infty[e^{-i(a+\varepsilon)z}I]\neq0$.
\end{theorem}
The following generalization of this result was obtained in \cite{MP1}.
\begin{theorem}\textnormal{(}\cite[Theorem 5.4]{MP1}\textnormal{)} Let $\Theta$ be a meromorphic inner function satisfying $\varphi_\Theta'\in L^\infty(\mathbb{R})$. Then for any meromorphic inner function $I$ we have: if $N^+[\overline{\Theta}I]\neq 0$, $a>0$, then for any $\varepsilon >0$
$N^\infty[e^{-i\varepsilon z}\overline{\Theta}I]\neq0$.
\end{theorem}

\subsection{Unifying theorem.} The family of disjoint intervals  $\{I_l\}$ is called {\it short} if
$$\sum_l\frac{|I_l|^2}{1+\dist^2(0,I_l)}<\infty.$$
Otherwise we call the family {\it long}. We have used the uniform version of this property (see \eqref{unshort}) already.

The next definiton is well known and comes from "the sunrise"  lemma (see \cite{R}).
 Suppose that a continuous function $\gamma$ on $\mathbb{R}$ satisfies
\begin{equation}
\gamma(-\infty)=+\infty,\qquad\gamma(+\infty)=-\infty.
\label{infty}
\end{equation}
The family $BM(\gamma)$ is defined as the collection of the components of the open set $\{\gamma^*\neq\gamma\}$, where 
$$\gamma^*(x)=\max_{[x,+\infty)}\gamma.$$

\begin{theorem}\textnormal{(}\cite[Theorem 5.8]{MP1}\textnormal{)}
Suppose $\inf\gamma'>-\infty$.
\begin{enumerate}
\begin{item}
If $\gamma\notin\eqref{infty}$, or if $\gamma\in\eqref{infty}$ but the family $BM(\gamma)$ is long, then 
$$\text{for any } \varepsilon>0,\quad N^+[e^{i\varepsilon z}e^{i\gamma}]=0.$$
\end{item}
\begin{item}
If $\gamma\in\eqref{infty}$ and $BM(\gamma)$ is short, then 
$$\text{for any } \varepsilon>0,\quad N^+[e^{-i\varepsilon z}e^{i\gamma}]\neq 0.$$
\end{item}
\end{enumerate}
\end{theorem}
The first statement corresponds to the "the Second BM-theorem", and statement $(ii)$ to the so-called "little multiplier theorem", see \cite{HJ, Koo1}.

\subsection{Tempered growth of the argument} 
We conclude this section by the theorem which can be applied to phase functions with tempered growth $\phi'(x)\leq C(1+|x|)^N$. 

If $\kappa\geq0$, then we say that $\gamma$ is {\it $(\kappa)$-almost decreasing} if $\gamma$ satisfies \eqref{infty} and
$$\sum_{I\in BM(\gamma)}\frac{|I|^2}{1+\dist^2(0,I)}<\infty.$$
If $\kappa=0$ we arrive at the usual definition of a system of {\it short} intervals.
\begin{theorem}\textnormal{(}\cite[Theorem A]{MP2}\textnormal{)}
Let $\kappa\geq0$, and let $\gamma$ and $\phi$ be smooth functions on $\mathbb{R}$ such that
$$\gamma'(x)\geq-C(1+|x|)^\kappa,\quad \phi'(x)\geq C|x|^\kappa,\quad C>0, x\rightarrow\infty.$$
\begin{enumerate}
\begin{item}
If $\gamma$ is not $(\kappa)$-almost decreasing, then $N^+[e^{i\gamma}e^{i\varepsilon\phi}]=0$ for all $\varepsilon>0$.
\end{item}
\begin{item}
If $\gamma$ is $(\kappa)$-almost decreasing, then $N^p[e^{i\gamma}e^{-i\varepsilon\phi}]\neq0$ for all $\varepsilon>0$ and all $p<1\slash3$.
\end{item}
\end{enumerate}
\end{theorem}

\section{Concluding Remarks}
\begin{remark}
The radius of completeness $R(\Lambda)$ (see, Subsection \ref{secondBM}) can be expressed in terms of Toeplitz operators. Namely,
$$R(\Lambda)=\inf\{a: \Ker(T_{B_\Lambda e^{-2aiz}})\neq0\}.$$
\end{remark}

In \cite{MP2} the following generalization of this quantity was studied:
$$R(J, \Theta):=\inf\{a: \Ker(T_{J\overline{\Theta}^a})\neq0\},$$
where $J$ and $S$ are meromorphic inner functions. This is {\it the generalized radius of completeness}. In some cases it equals to a corresponding density, see \cite[Theorem B]{MP2}, and our Subsection \ref{secondBM}.

Theorem \ref{borth} can be proved using Toeplitz kernel approach but we preferred to give a straightforward proof which uses only the Hadamard three circles theorem.

\end{document}